\def\ocirc#1{\ifmmode\setbox0=\hbox{$#1$}\dimen0=\ht0
    \advance\dimen0 by1pt\rlap{\hbox to\wd0{\hss\raise\dimen0
    \hbox{\hskip.2em$\scriptscriptstyle\circ$}\hss}}#1\else
    {\accent"17 #1}\fi}

\documentclass{amsart}
\usepackage{amsfonts}
\usepackage{amssymb}
\usepackage{amsmath}
\usepackage[mathcal]{euscript}
\usepackage{mathrsfs}
\usepackage{graphicx}
\usepackage{enumerate}
\usepackage{color}

 \newtheorem{thm}{Theorem}

\newtheorem{lem}[thm]{Lemma}
\newtheorem{cor}[thm]{Corollary}
\newtheorem{exmp}[thm]{Example}

\newtheorem{rem}[thm]{Remark}

\DeclareMathOperator{\dist}{dist}
\DeclareMathOperator{\Per}{Per}
\DeclareMathOperator{\Aper}{Aper}
\DeclareMathOperator{\Tran}{Tran}
\DeclareMathOperator{\Prox}{Prox}
\DeclareMathOperator{\Rec}{Rec}

 \newcommand{\Zp}{{\mathbb{N}_0}}
 \newcommand{\Cir}{\mathbb{S}^1}

 \newcommand{\orbp}{\text{Orb}^+}
 
 \newcommand{\eps}{\varepsilon}
 
 \newcommand{\ra}{\rightarrow}

 \newcommand{\F}{\mathscr{F}}
 
 \newcommand{\B}{\mathscr{F}_{inf}}

 \newcommand{\N}{\mathbb{N}}
 \newcommand{\Z}{\mathbb{Z}}

 \newcommand{\set}[1]{\left\{#1\right\}}

\newcommand{\disj}[2]{#1 \:\bot\: #2}

\numberwithin{enumi}{thm}

\begin{document}


\title[On weak product recurrence]{On weak product recurrence and synchronization of return times}

\author[Piotr Oprocha]{Piotr Oprocha}
\address[P. Oprocha]{AGH University of Science and Technology\\
Faculty of Applied Mathematics\\
al. A. Mickiewicza 30, 30-059 Krak\'ow,
Poland\\ -- and --\\Institute of Mathematics\\ Polish Academy of Sciences\\ ul. \'Sniadeckich 8, 00-956 Warszawa, Poland} \email{oprocha@agh.edu.pl}

\author[Guohua Zhang]{Guohua Zhang*}
\address[G.~H.~Zhang]{School of Mathematical Sciences and LMNS, Fudan University, Shanghai 200433, China}

\email{chiaths.zhang@gmail.com}

\begin{abstract}
The paper is devoted to study of product recurrence. First, we prove that notions of $\F_{ps}-PR$ and $\F_{pubd}-PR$ are exactly the same as product
recurrence, completing that way results of [P. Dong, S. Shao and X. Ye, \emph{Product recurrent properties, disjointness and
weak disjointness}, Israel J. Math.],
 and consequently, extending the characterization of return times of distal points which originated from works of Furstenberg.
We also study the structure of the set of return times of weakly mixing sets. As a consequence, we obtain new sufficient conditions for $\F_{s}-PR$ and also find a short proof that weakly mixing systems are disjoint with all minimal distal systems (in particular, our proof does not involve Furstenberg's structure theorem of minimal distal systems).
\end{abstract}



\thanks{*Corresponding author (chiaths.zhang@gmail.com)}

\maketitle

\section{Introduction}

There are two important characterizations related to distality, both observed by Furstenberg more than 30 years ago.
If a point is distal, then it is recurrent in pair with any recurrent point in any dynamical system \cite{FurBook}
and if a minimal dynamical system is distal then it is disjoint with any weakly mixing system \cite{FurDA}.
It is worth emphasizing that later, a full characterization of flows disjoint with all distal flows was provided by Petersen in \cite{Petersen}.
While proofs of both above mentioned characterizations are not that long, their proofs highly rely on other, even more important and highly nontrivial results.
Namely, first of them uses Hindman's theorem on finite partitions of IP-sets, while second can be obtained as a consequence of an algebraic characterization of distal flows, proved by Furstenberg in \cite{STDF}.

The above characterizations on synchronization of return times of a point with return times of distal points gave motivation to two directions of research.
First of them asks about synchronization of return times between
different types of recurrent points and second asks about disjointness between specified classes of systems. Both these questions lead to partial classifications and hard open problems \cite{Ye_WPR}.
It is also worth emphasizing that studies on the above topics were
very influential and important for topics which at the first sight do not seem to be related very much.
For example, Blanchard characterized in \cite{Bla} systems disjoint with flows with zero entropy giving that way a good motivation
for introducing entropy pairs and systems with uniform positive entropy (u.p.e.) which are two fundamental notions in local entropy theory
which were used later for deep and insightful investigations on topological entropy (cf. \cite{GY} for a more story of local entropy theory). In general, results on synchronization of trajectories
can be of wide use (e.g. they can help to simplify some arguments in proofs etc.). For example, analysis of return or transfer times of points
lead to a simple proof that distal point is always minimal or that point minimal or recurrent for a dynamical system $(X,f)$ is, respectively, minimal and recurrent
for $(X,f^n)$ for every $n$, etc.

If in place of recurrence in pair with any recurrent point we demand recurrence in pair with points in a smaller class of dynamical systems, it can lead to
a wider class of points than the class of all distal points. For example, Auslander and Furstenberg  in \cite{AusFur} asked about points which are recurrent in pair with any minimal point. While there is no known full characterization of points with this property,
it was proved in \cite{Ott} that class of such points is much larger than distal points, in particular it contains many points which are not minimal.
Other sufficient conditions for this kind of product recurrence were provided in \cite{Ye_WPR} and \cite{PO_AIF}.
Moreover, \cite{Ye_WPR} defines product recurrence in terms of Furstenberg families (i.e. upward hereditary sets of subsets of $\N$), which
is a nice tool for a better classification of product recurrence. It is worth emphasizing that the concepts of \cite{Ye_WPR} are not artificial, since it is possible almost
immediately to relate these new types of product recurrence with some older results on disjointness.

The present paper completes some previous studies on recurrence and product recurrence from \cite{Ye_WPR} and \cite{PO_AIF}.
First, we prove that notions of $\F_{ps}-PR$ and $\F_{pubd}-PR$ are exactly the same as product
recurrence, that is, if return times of a point can be synchronized with points returning with a piecewise syndetic set of times, then it can be synchronized with any recurrent point. This provides another condition to the list of conditions equivalent to distality as first provided by Furstenberg in \cite[Theorem 9.11]{FurBook}, and next extended by many authors (e.g. see \cite{Ye_WPR}).
Next, we analyze synchronizing properties of points in weakly mixing sets, which allow us to show that any weakly mixing set with dense distal points contains
a residual subset of $\F_{s}-PR$ points which are not distal (this is a question left open in \cite{PO_AIF}) and also to prove Furstenberg's result on disjointness between
weakly mixing systems and minimal distal systems without referring to Furstenberg's structure theorem of distal flows. This is especially nice, since now both results of Furstenberg
mentioned in the first paragraph of this introduction can be obtained using only Hindman's theorem plus some topological arguments. This even more bonds these two results together.


\section{Preliminaries}
In this section, we will provide some basic definitions used later in this paper.
The reader is encouraged to refer to the books \cite{Aki93, Aki97, HS} for more details.

Denote by $\mathbb{N}$ ($\Zp, \mathbb{Z}, \mathbb{R}$, respectively) the set of all positive integers (non-negative integers, integers, real numbers, respectively).
A set $A \subset \N$ is an \emph{IP-set} if there exists a sequence $\set{p_i}_{i=1}^\infty \subset \N$
such that $A$ consists exactly of numbers $p_i$ together with all finite sums
$$
p_{n_1} + p_{n_2} + \cdots + p_{n_k} \text{ with } n_1 < n_2 < \cdots < n_k, k\in \N.
$$

%
%
%

\subsection{Basic notions in topological dynamics}
By a \emph{(topological) dynamical system}, or \emph{TDS} for short, we mean a pair consisting of a compact metric space $(X,d)$ and a continuous map $f\colon X\ra X$.
We denote the \emph{diagonal} in $X\times X$ by $\Delta_2 (X)=\set{(x,x): x\in X}$.

The \emph{(positive) orbit of $x\in X$ under $f$} is the set $\orbp (x, f)= \set{f^n (x)\; :\; n\in \N}$. We denote
$N_f(x,A)=\set{n\in \N : f^n(x)\in A}$ and  similarly $N_f(A,B)= \{n\in \N : f^n(A)\cap B\neq \emptyset\}$.

We say that $x$ is 
a {\it periodic point of $(X, f)$} if $f^n (x)= x$ for some $n\in \mathbb{N}$; a {\it recurrent point of $(X, f)$} if $N_f(x,U)\neq \emptyset$ for any open set $U\ni x$; a {\it transitive point of $(X, f)$} if $\overline{\orbp (x, f)}= X$.
Denote by $\Per (X, f)$ ($\Rec (X, f)$ and $\Tran (X, f)$, respectively) the set of all periodic points (recurrent points and transitive points, respectively) of $(X, f)$.

Recall that $(X, f)$ is {\it transitive} if $N_f(U,V)\neq \emptyset$ for any non-empty open sets $U$ and $V$ and {\it minimal} if $\Tran (X, f)= X$.
A point $x\in X$ is \emph{minimal} or \emph{uniformly recurrent} if $(\overline{\orbp(x,f)},f)$ is minimal. A dynamical system $(X, f)$ is an \emph{M-system} if it is transitive and the set of all minimal points is dense in $X$.

A pair of points $x,y\in X$ is \emph{proximal} if there exists an increasing sequence
$\set{n_k}_{k=1}^\infty\subset \N$ such that $\lim_{k\ra \infty} d(f^{n_k}(x),f^{n_k}(y))=0$.
A point $x$ is \emph{distal} if it is not proximal to any point in its orbit closure other than itself.
It was first observed by Auslander and Ellis that any point is proximal to a minimal point from its orbit closure and so any distal point is minimal \cite{Aus, Ellis, FurBook}.
The set of all proximal pairs of $(X,f)$ is denoted by $\Prox(f)$ and a proximal cell of $x\in X$ is denoted by $\Prox(f)(x)=\set{y\in X\; : \; (x,y)\in \Prox(f)}$.
We say that $(X,f)$ is \emph{distal} if every point $x\in X$ is a distal point. Equivalently, it is to say that there is no proper proximal pair in $X$ (i.e. $\Prox(f)=\Delta_2 (X)$).

Let $f,g$ be two continuous surjective maps acting on compact metric spaces $X$ and $Y$ respectively.
We say that a non-empty closed set $J\subset X\times Y$  is \emph{a joining} of $(X,f)$ and $(Y,g)$ if it is invariant (for the product map $f\times g$) and its projections onto the first and second coordinates are $X$ and $Y$ respectively.
If each joining is equal to $X\times Y$ then we say that $(X,f)$ and $(Y,g)$ are \emph{disjoint} and denote this fact by $\disj{(X,f)}{(Y,g)}$ or simply by $\disj{f}{g}$.

\subsection{Families and product recurrence}

\emph{A (Furstenberg) family} $\F$ is a collection of subsets of $\N$ which is \emph{upwards hereditary},
that is:
$$
F_1 \in \F \text{ and } F_1 \subset F_2
\quad \Longrightarrow \quad F_2 \in \F.
$$
A family $\F$ is \emph{proper} if $\N \in \F$ and $\emptyset \notin \F$ and is a \emph{filter} if it is a proper family closed under finite intersections (equivalently $A\cap B\in \F$ for every $A,B\in \F$). Among filters there is an important class which is maximal with respect to set inclusion.
Any such a filter is called an \emph{ultrafilter}. Note that by Zorn's Lemma every
filter is contained in some ultrafilter. We denote by $\beta \N$ the set of all ultrafilters of $\N$.

 For any $n\in \N$ we can define its \emph{principal ultrafilter} $e(n)=\set{A\subset \N : n\in A}$, therefore
we can write $\N\subset \beta \N$ by the natural identification.
Given $A\subset \N$ we set $\hat{A}=\set{p\in \beta \N : A\in p}$, and then we can define a topology on $\beta \N$ which has $\set{\hat{A} : A\subset \N}$ as its basis. It can be proved that $\beta \N$, equipped with the above introduced topology, coincides with {S}tone-\v {C}ech compactification of the discrete space $\N$, in particular, $\beta \N$ is
a compact Hausdorff space and the set $\set{e(n) : n\in \N}$ is a dense subset of $\beta\N$ whose points are precisely the isolated points of $\beta \N$ (e.g. see \cite[Theorem~3.18 and Theorem~3.28]{HS}).
For any $p,q\in \beta\N$ we define
$$
p + q = \set{A\subset \N : \set{n\in \N : -n + A\in q} \in p},
$$
and we call $p$ an \emph{idempotent} if $p+ p= p$.
It can be proved (e.g. see \cite[Chapter 4]{HS}) that $(\beta\N, +)$ is a right topological semigroup with $\N$ contained in its topological center, in the sense that $e(n)+e(m)=e(n+m)$ for all $n, m\in \N$ and
\begin{enumerate}[(i)]
\item for each $q\in \beta \N$, the map $\rho_q:\beta\N \ni p \mapsto p+q\in \beta \N$ is continuous,
\item for each $n\in \N$, the map $\lambda_n:\beta\N \ni p \mapsto e (n)+p\in \beta \N$ is continuous.
\end{enumerate}
For a more detailed exposition on {S}tone-\v {C}ech compactifications, in particular the case of $\beta \N$, the reader is referred to the book \cite{HS} by Hindman and Strauss.

Recall that a set $A\subset \N$ is \emph{thick} if for every $n>0$ there is an $i\in \N$ such that $\set{i,i+1,\cdots, i+n}\subset A$; is \emph{syndetic} if $A$ has a bounded gap, that is, there exists $N\in \N$ such that $[i,i+N]\cap A\neq \emptyset$ for each $i\in \N$. We denote by $\B$, $\F_t$ and $\F_s$ the family of all infinite subsets, thick subsets and syndetic subsets of $\N$, respectively.
It is direct to see that each thick subset intersects all syndetic subsets. We denote by $\F_{ps}$ the family of all \emph{piecewise syndetic} sets, that is sets that can be obtained as the intersection of a thick set and a syndetic set.
We denote by $\F_{pubd}$ the family of sets with \emph{positive upper Banach density}, that is sets $F\subset \N$ such that
$$
\limsup_{n-m\to\infty}\frac{\# (F\cap\{m,m+1,\cdots,n\})}{n-m+1}>0,
$$
where as usual $\# A$ denotes the cardinality of a set $A$.

Families may be used to state definitions of recurrent points with prescribed types of the set of return times.
Namely, for a family $\F$  and $x\in X$, we say that $x$ is \emph{$\F$-recurrent} if $N_f(x,U)\in \F$ for any open neighborhood $U$ of $x$. Note that a point is recurrent exactly when it is $\B$-recurrent
and is minimal when it is $\F_s$-recurrent. For an interesting exposition on recurrence properties expressed in
terms of families the reader is referred to the book \cite{Aki97} by Akin.

A recurrent point $x$ in a dynamical system $(X,f)$ is \emph{product recurrent} if given any recurrent point $y$ in any dynamical system $(Y,g)$ the pair $(x,y)$ is recurrent for the product system $(X\times Y, f\times g)$. If we demand above condition only for $y$ which is uniformly recurrent then we say that $x$ is \emph{weakly product recurrent}. Again, let $\F$ be a family. A recurrent point $x$ in a dynamical system $(X,f)$ is \emph{$\F$-product recurrent} (\emph{$\F$-PR} for short) if for any $\F$-recurrent point $y$ in any dynamical system $(Y,g)$, the pair $(x,y)$ is recurrent for $(X\times Y, f\times g)$.
Thus, a recurrent point is product recurrent if and only if it is $\F_{inf}$-PR, and is weakly product recurrent if and only if it is $\F_{s}$-PR.

It is well known that a point is product recurrent if and only if it is distal \cite{FurBook} and it was recently proved in \cite{Ott} that there are weakly product recurrent points which are not distal (in fact, they form a much wider class of points).
Properties of product recurrence were studied independently in \cite{Ye_WPR} and \cite{PO_AIF}, where some further necessary conditions for product recurrence were obtained.
It is obvious that
$$
\F_{inf}-PR \quad \Longrightarrow \quad \F_{pubd}-PR \quad \Longrightarrow \quad \F_{ps}-PR.
$$
The question whether any of the above implications can be reverted was left open in \cite{Ye_WPR}. We will answer this question later in this paper.

\subsection{Weakly mixing sets}

A closed set $A\subset X$ containing at least two points is a \emph{weakly mixing set of order $n$} if for any
choice of open subsets $V_1, U_1, \cdots, V_n, U_n$
of $X$ with $A \cap U_i \neq \emptyset$, $A \cap V_i \neq \emptyset$, $i=1,\cdots,n$, there exists $k>0$ such that $f^k(V_i\cap A) \cap U_i \neq \emptyset$
for each $1 \leq i \leq n$.
If $A$ is weakly mixing of order $n$ for all $n\geq 2$, then we say that $A$ is \emph{weakly mixing of all orders}, or simply \emph{weakly mixing}.

The idea of weakly mixing sets comes from \cite{BH} where it was studied firstly in considerable detail.
In particular, it was
shown that every system with positive entropy contains weakly mixing sets \cite{BH}. While, we are looking at the a priori weaker property of weak mixing of order $n$ and especially 2 in \cite{PZ, PZpreprint}. It was
proved in \cite{PZpreprint} that for each $n\geq 2$ there exists a minimal system containing weakly mixing sets of order $n$ but without
weakly mixing sets of order $n+1$ (in \cite{PZ} such an example was constructed only for the particular case of $n=2$).

\begin{rem}
The definition of weakly mixing set (of order $n$) introduced here is a little more restrictive (i.e. more conditions are put on $A$)
than that in \cite{PZ}, where it was introduced first.
\end{rem}

It can be proved easily that each weakly mixing set of order $2$ (as introduced here) is perfect \cite{PZ}.
If the whole $X$ is a weakly mixing set then we say that $(X,f)$ is weakly mixing, which equivalently means that $(X\times X,f\times f)$ is transitive. Furstenberg proved that for an invariant subset, weak mixing of order 2, implies weak mixing of all orders \cite[Proposition II.3]{FurDA}
(i.e. on invariant subsets, these notions coincide).

%
%

\section{Piecewise syndetic product recurrence} \label{section}

As we mentioned earlier, the question whether any of the implications below can be reverted was left open in \cite[Section~5.3]{Ye_WPR}:
$$
\F_{inf}-PR \quad \Longrightarrow \quad \F_{pubd}-PR \quad \Longrightarrow \quad \F_{ps}-PR.
$$
In this section we will prove that all the above properties are equivalent, and then as a corollary of this equivalence we can extend an important characterization of distal points from \cite[Theorem 9.11]{FurBook}.

To settle down the question we need the following notions.
A set $S\subset \N$
is a \emph{dynamical syndetic set} if there exists a minimal dynamical system $(X, f)$ with a minimal point $x\in X$ and an open neighborhood $U$ of $x$ such that $S= N_f (x,U)$.
A set $J\subset \N$ is an \emph{md-set} if there exists an M-system $(Y, g)$ with a transitive point $y\in Y$ and a neighborhood $V$ of $y$ such that $J= N_g (y, V)$.

It was proved in \cite[Proposition 3.3]{Ye_WPR} that every thick set contains an md-set.
The following lemma shows a similar result for a finer structure.

\begin{lem} \label{1202211736}
Let $S_1$ and $S_2$ be a dynamical syndetic set and a thick set, respectively. Then $S_1\cap S_2$ contains an md-set
$N_g(z,W)$ defined by an $\F_{ps}$-recurrent and transitive point $z$ in an M-system $(Y,g)$ and an open neighborhood $W$ of $z$.
\end{lem}
\begin{proof}
From the definition, there exists a minimal dynamical system $(X, f)$ with a minimal point $x\in X$ and an open neighborhood $U$ of $x$ such that $S_1= N_f (x, U)$.
Moreover, as $S_2$ is a thick set, there is an increasing sequence $n_j$ such that
$n_{j+1}-n_j>j^2$ and $B_j\subset S_2$ with $B_j= [n_j,n_j+j)\cap \N$. Note that sets $B_j$ are pairwise disjoint.
Let us renumerate sequence $\set{B_j}_{j=1}^\infty$ creating a double indexed family $\set{B_j^{(i)}}_{i,j=1}^\infty$ with the property that
$k<s$ whenever $k\in B^{(i)}_1$, $s\in B^{(i+1)}_1$ or $k\in B^{(i)}_j$, $s\in B^{(i)}_{j+1}$ for some indices $i,j=1,2,\cdots$.

Now we will perform an inductive construction of points $z_x^{(i)}\in \Sigma_2$ and sets $A^{(i)}\subset \N$, where
by $\Sigma_2$ we denote the two-sided full shift over the alphabet $\set{0,1}$ (together with the left shift transformation $\sigma$). If $u,v\in \set{0, 1}^k$ then we write $u \preccurlyeq v$ if $u[i]\leq v[i]$ for every $0\leq i < k$. We can extend easily this relation onto infinite sequences and next to bi-infinite sequences $y\in \Sigma_2$ such that $y[i]=0$ for $i<0$. Define a point $z_x\in \Sigma_2$ by putting $z_x[i]=0$ if $i<0$, $z_x[i]=0$ if $f^i(x)\not\in U$ for some $i\geq 0$ and $z_x[i]=1$ in all other situations.

Now, let us put $z_x^{(1)}[i]=1$ when $i=0$ or when $f^i(x)\in U$ and $i\in \bigcup_{j=1}^\infty B_j$. For all other values of $i$ we put $z_x^{(1)}[i]=0$. Then $z_x^{(1)}\preccurlyeq z_x$. Define
$$A^{(1)}=\set{i\in \N: f^i(x)\in U, i\in \bigcup_{j=1}^\infty B^{(1)}_j}.$$
The point $x$ is minimal, so clearly $A^{(1)}$ is piecewise syndetic and $z_x^{(1)}[k]=1$ for every $k\in A^{(1)}$.

Assume that for some $m\geq 1$ we have constructed points $z_x^{(m)}\preccurlyeq z_x^{(m-1)} \preccurlyeq \cdots \preccurlyeq  z_x^{(1)}$ and piecewise syndetic sets $A^{(1)},\cdots, A^{(m)}$ with the following additional properties:
\begin{enumerate}
\item\label{zxind:c1} $z^{(s)}_x[k]=z^{(s+1)}_x[k]$ provided that $k\not\in \bigcup_{j=1}^\infty B_j^{(s+1)}$, where $s<m$;
\item\label{zxind:c2} $(k-s,k+s)\cap \Z \subset  \bigcup_{j=1}^\infty B_j^{(s)}$ for every $k\in A^{(s)}$, where $s=1,\cdots,m$; and
\item\label{zxind:c3} $(z_x^{(s)})_{(-s, s)}=(z_x^{(s)})_{(k-s, k+s)}$ for every $k\in A^{(s)}$, where $s=1,\cdots,m$.
\end{enumerate}

Now we are going to construct $z_x^{(m+1)}$ and $A^{(m+ 1)}$.

First we put $z_x^{(m+1)}[k]=z_x^{(m)}[k]$ for every $k\not\in \bigcup_{j=1}^\infty B_j^{(m+1)}$.
On other positions we will copy only a part of symbols $1$ from $z_x^{(m)}$, changing into $0$ some of them. Note that by the construction we have
$z_x^{(m)}[k]=z_x^{(1)}[k]$ for every $k\in \bigcup_{s>m}\bigcup_{j=1}^\infty B_j^{(s)}$.

 There is an integer $t$ such that if $a\in B^{(m+1)}_j$, $b\in B^{(m+1)}_{j+1}$ for some $j\geq t$ then $b-a>3m$. In particular, if $i-m,i+m\in \bigcup_{j=t}^\infty B_j^{(m+1)}$ for some $i\in \N$ then
there is $j\ge t$ such that $[i-m,i+m]\cap \N \subset B_j^{(m+1)}$.
We choose an open set $V\ni x$ such that, for $0\leq j \leq m$,
$f^j(x)\in U$ implies $f^j(V)\subset U$.
Obviously the following set is piecewise syndetic:
$$
A=\set{k\in \N: f^k(x)\in V \text{ and }k-m\in \bigcup_{j=t}^\infty B_j^{(m+1)}, k+m\in \bigcup_{j=t}^\infty B_j^{(m+1)}}.
$$
We can remove elements in $A$ if necessary, keeping it piecewise syndetic and at the same time ensuring that $|a-b|>3m$ whenever $a,b\in A$ are distinct. Denote by $A^{(m+1)}$ the set $A$ after this modification. This modification ensures that if we put $(z_x^{(m+1)})_{[k,k+m]}=(z_x^{(m)})_{[0,m]}$ for
all $k\in A^{(m+1)}$ and put $z_x^{(m+1)} [i]=0$ for all $i\in \bigcup_{j=1}^\infty B_j^{(m+1)}\setminus \bigcup_{i\in A^{(m+1)}}[i,i+m]$
then $z_x^{(m+1)}$ is well defined (simply because $|a- b|> 3 m$ and so $[a-m,a+m]\cap [b-m,b+m]=\emptyset$ for distinct $a,b\in A^{(m+1)}$).

Now it remains to check that $z_x^{(m+1)}$ and $A^{(m+1)}$ satisfy all desired properties. Directly from the construction we have that $A^{(m+1)}$
is piecewise syndetic and conditions \eqref{zxind:c1}, \eqref{zxind:c2} are satisfied.
Note that $\min \bigcup_{j=1}^\infty B_j^{(m+1)}>m$ therefore $(z_x^{(m+1)})_{[0,m]}=(z_x^{(m)})_{[0,m]}$ which also gives \eqref{zxind:c3} with the help of the construction.
The only condition which remains is $z_x^{(m+1)}\preccurlyeq z_x^{(m)}$. Fix any $k\in A^{(m+1)}$. Then by the construction we have that
$(z_x^{(m+1)})_{[k,k+m]}=(z_x^{(m+1)})_{[0,m]}$.
Observe that $f^k(x)\in V$ and therefore by the definition of $V$, if $f^j(x)\in U$ for some $0\leq j\leq m$
then $f^{k+j}(x)\in U$. In other words $(z_x)_{[0,m]}\preccurlyeq (z_x^{(1)})_{[k,k+m]}$. Additionally by the construction and \eqref{zxind:c1} we have that $(z_x^{(m)})_{[k,k+m]}=(z_x^{(1)})_{[k,k+m]}$. Combining all these facts together we obtain the following:
\begin{eqnarray*}
(z_x^{(m+1)})_{[k,k+m]}&=&(z_x^{(m+1)})_{[0,m]}=(z_x^{(m)})_{[0,m]}\\
&\preccurlyeq& (z_x)_{[0,m]}\preccurlyeq (z_x^{(1)})_{[k,k+m]}=(z_x^{(m)})_{[k,k+m]}.
\end{eqnarray*}
But for all $i\in \Z \setminus \bigcup_{j\in A^{(m+1)}}[j,j+m]$ we also have $z_x^{(m+1)}[i]\leq z_x^{(m)}[i]$ from the construction, and so $z_x^{(m+1)}\preccurlyeq z_x^{(m)}$
which completes the induction.

By the definition obviously $\set{z_x^{(m)}}_{m= 1}^\infty$ is a Cauchy sequence in $\Sigma_2$, so the limit $z=\lim_{m\to \infty} z_x^{(m)}$ is well defined.
Furthermore, the only modifications of $z_x^{(s)}$ in the further steps of induction are done on the set $Q_s=\bigcup_{i=s+1}^\infty \bigcup_{j=1}^\infty B_j^{(i)}$ and so $z[k]=z_x^{(s)}[k]$ for every $k\not\in Q_s$.
In particular, for every $i\in A^{(s)}$ we have
$$
z_{(-s,s)}=(z_x^{(s)})_{(-s,s)}=(z_x^{(s)})_{(i-s,i+s)}=z_{(i-s,i+s)}
$$
which equivalently means that
$N_\sigma(z,W)\supset A^{(s)}$, where $\sigma$ is the shift transformation over $\Sigma_2$ and $W$ is the cylinder set $\set{q\in \Sigma_2 : q_{(-s,s)}=z_{(-s,s)}}.$ This proves that $z$ is $\F_{ps}$-recurrent.
Denote by $Z$ the closure of the (positive) orbit $\orbp (z, \sigma)$ of $z$ under the shift transformation.
It can be proved (e.g. see \cite[Lemma 2.1]{Huang}) that $(Y, g)$ is an M-system if and only if there is a transitive point $y\in Y$ such that $N_g (y, W)\in \F_{ps}$ for any neighborhood $W$ of $y$. Thus $(Z, \sigma)$ is an M-system and so from the construction $S_1\cap S_2$ contains an md-set $N_\sigma (z, W_1)$, where $W_1$ is the cylinder set $\set{q\in \Sigma_2: q [0]= 1}$.
\end{proof}

%
%
%
%

By characterization from \cite{FurBook}, a point $x$ is distal if and only if $x$ is $\F_{inf}-PR$, so to answer our question it is enough to prove the following theorem.

\begin{thm}\label{thm:fps}
If $x$ is $\F_{ps}$-PR then it is distal.
\end{thm}
\begin{proof}
First note that $x$ is minimal by \cite[Theorem~3.4]{Ye_WPR}. Let $X$ be the closure of the (positive) orbit of $x$ under the associated map $f$.
 Assume on the contrary that $x$ is not distal, which means that there is a minimal point $y\in X\setminus \set{x}$ such that
the pair $(x,y)$ is proximal. Let $\eps> 0$ and $U, V$ be two open sets such that $y\in U$, $x\not\in \overline{V}$ and $B (\overline{U}, \eps)\subset V$.

The pair $(x,y)$ is proximal, and so there is an increasing sequence $n_j$ such that
$n_{j+1}-n_j>j^2$ and $d(f^i(x),f^i(y))<\frac{\eps}{j}$ for every $i\in [n_j,n_j+j)$. We denote $S_1= N_f (y, U)$ and let $S_2= \bigcup_{j=1}^\infty [n_j, n_j+ j)\cap \N$.
Then $S_1$ is a dynamical syndetic set and $S_2$ is a thick set.

By Lemma~\ref{1202211736} the intersection $S_1\cap S_2$ contains a set $N_g (z, W)$, where $z$ is a transitive and $\F_{ps}$-recurrent point from some M-system $(Z, g)$, and $W$ is a neighborhood of $z$.
But $(x,z)\in (X\setminus \overline{V})\times W$ and for any $k>0$ if $g^k (z)\in W$ then $f^k (y)\in U$ and $d (f^k (x), f^k (y))< \eps$, which implies that $f^{k}(x)\in V$. This shows that $(x,z)$ is not recurrent, a contradiction to the assumption that $x$ is $\F_{ps}-PR$, which ends the proof.
\end{proof}

So far, we have obtained by Theorem~\ref{thm:fps} the equivalence of the properties $\F_{inf}-PR, \F_{pubd}-PR$ and $\F_{ps}-PR$. Combining this equivalence with recent results from \cite{LJFM}, we obtain the following characterization, which is an extension of the classical characterization of distality in \cite[Theorem 9.11]{FurBook}.
While it is not visible here,  the proof strongly relies on the structure of the set idempotent in
$\beta \N$.  For a more detailed exposition on this important topic see for example \cite{BerD, Ellisbook, Glasner, HS, LJFM}.

Fix a family $\F$ and any sequence $\set{x_n}_{n=1}^\infty\subset X$. We say that $z\in X$ is an \emph{$\F$-limit}
of $\set{x_n}_{n=1}^\infty$ if for every open neighborhood $U$ of $z$ the set $\set{n \in \N :
x_n \in U} \in \F$. Observe that if $\F$ is a filter then there exists at most one such a point $z$.
 It can be proved that every ultrafilter $p\in \beta \N$ has \emph{Ramsey Property}, that is, if
$A\cup B\in p$ then either $A\in p$ or $B\in p$. Therefore, it is not hard to show that for any sequence $\set{x_n}_{n=1}^\infty\subset X$ and any ultrafilter $p\in \beta \N$
there exists a unique point $z\in X$ such that $z$ is a $p$-limit of $\set{x_n}_{n=1}^\infty$. In what follows, for any $x\in X$ and any $p\in \beta \N$ we will denote by
$p x$ the $p$-limit of the sequence $\set{f^n(x)}_{n=1}^\infty$.
By the definition of the Stone-\v{C}ech compactification, any function $\N \ni n \mapsto x_n\in X$ can be extended to be a continuous function $\beta\N \ni p \mapsto x_p \in X$.
But it is also not hard to prove that for every $p\in \beta\N$ the point $x_p\in X$ is just the $p$-limit of the sequence $\set{x_n}_{n=1}^\infty$ (e.g. see \cite[Corollary~3.49.1]{HS}).

A family $\F$ is a \emph{filterdual} if its dual family $k\F$ is a filter, where $k\F= \{F\subset \N: \N\setminus F\not\in \F\}$.
It can be proved that both $\F_{pubd}$ and $\F_{ps}$ are filterduals (e.g. see \cite{Aki97}) and that both sets $h(\F_{pubd})$ and $h(\F_{ps})$ are
closed subsemigroups of $(\beta\N, +)$ where $h(\F)=\set{p\in \beta\N : p\subset \F}$ (e.g. see \cite{LJFM}). But then with respect to these families
we can apply the following equivalent characterization (it is a shortened version of a list of equivalent conditions in \cite[Theorem~4.4]{LJFM}):

\begin{thm}\label{JL:recurence}
Let $\F$ be a filterdual and suppose that $h(\F)$ is a subsemigroup of $(\beta \N, +)$.
Then $x$ is $\F$-recurrent if and only if there exists an idempotent $p\in h(\F)$ such that $px=x$.
\end{thm}

 Now we are ready to prove above announced extension of Furstenberg's theorem.

\begin{thm} \label{distal}
The following statements are equivalent:
\begin{enumerate}[(1)]

\item $x$ is distal, \label{first}

\item $(x, y)$ is recurrent for any recurrent point $y$ of a system $(Y, g)$, \label{second}

\item $(x, y)$ is $\F_{pubd}$-recurrent for any $\F_{pubd}$-recurrent point $y$ of a system $(Y, g)$, \label{third}

\item $(x, y)$ is $\F_{ps}$-recurrent for any $\F_{ps}$-recurrent point $y$ of a system $(Y, g)$, \label{forth}

\item $(x, y)$ is minimal for any minimal point $y$ of a system $(Y, g)$. \label{fifth}
\end{enumerate}
\end{thm}
\begin{proof}
The equivalence of
\eqref{first}, \eqref{second} and \eqref{fifth} comes from \cite[Theorem 9.11]{FurBook}.

Implications $\eqref{third}\Longrightarrow \eqref{first}$ and $\eqref{forth}\Longrightarrow \eqref{first}$ follow from the equivalence of the properties $\F_{inf}-PR, \F_{pubd}-PR$ and $\F_{ps}-PR$.

Now we prove the implication of $\eqref{first}\Longrightarrow \eqref{third}$.
Fix any $\F_{pubd}$-recurrent point $y$ of a system $(Y, g)$. Then by Theorem~\ref{JL:recurence} there exists an idempotent $p\in \beta \N$ such that $p\subset \F_{pubd}$ and $p y= y$.
But it is not hard to check, that if $p$ is an idempotent in $\beta \N$ and $x$ is distal then $px=x$ (e.g. see \cite[Proposition~3.17]{Ber}).
Since $p$ is a filter, for any neighborhood $U$ of $x$ and $V$ of $y$ we have that $\set{n\in \N : f^n(x)\in U}\cap \set{m\in \N : g^m(y)\in V}\in p \subset \F_{pubd}$, which implies that $(x, y)$ is $\F_{pubd}$-recurrent.

The proof of implication $\eqref{first}\Longrightarrow \eqref{forth}$ is identical to that of $\eqref{first}\Longrightarrow \eqref{third}$, so we leave it to the reader.
\end{proof}


Denote by \emph{$\F$-PR$_0$} the restriction in the definition of $\F$-PR considering recurrence in pair only with points $y$ from dynamical systems $(Y,g)$ with zero topological entropy.
Generally speaking, \emph{$\F$-PR$_0$} denotes product recurrence with respect to systems of zero entropy.

Observe that in the proof of Lemma~\ref{1202211736} the system generated by $z$ can have positive topological entropy. For example, if $x$ is from a minimal subshift $X$ over $\set{0,1}$ then it may happen that in the first step of the induction we will incorporate in $z$ arbitrarily long subwords of $x$, and therefore entropy of the subshift generated by $z$ will be at least as that of $X$. In fact, this point is not surprising, since it was shown in \cite{Ye_WPR} that there exists minimal systems with positive entropy (therefore not distal) such that every point in this system is $\F_{ps}$-PR$_0$ (an important aid when dealing with $\F$-PR$_0$ property, especially for construction of counterexamples, is provided by results of \cite{HPY} on disjointness with zero entropy systems).
In particular, another question from  \cite{Ye_WPR} about implication
$$
\quad \F_{ps}-PR_0 \quad \Longrightarrow \F_{pubd}-PR_0
$$
remains still open.

\section{Weak mixing and synchronization}

In this section we investigate properties of the set of transfer times of points in weakly mixing sets.
One of possible applications of our analysis is
another proof of an important result by Furstenberg \cite{FurDA} stating that every weakly mixing system is disjoint from all minimal distal systems.
As another application we show that a weakly mixing set with dense distal points contains
a residual subset of $\F_{s}-PR$ points which are not distal, which completes results of \cite{PO_AIF}.

 For this purpose, given sets $U, V\subset X$ we denote $C_f(U,V) = \set{n\in \N: f^n(U)\subset V}$.
Additionally, for any sequence of positive integers $p_1,p_2,\cdots$ and $n\geq 1$, denote by $S(p_1,\cdots,p_n)$ the set of all sums of subsequences of $p_1,\cdots, p_n$, that is,
$$
S(p_1,\cdots,p_n)=\set{p_{i_1}+\cdots +p_{i_k}\; : \; 1\leq k\leq n,\; i_1<\cdots <i_k}.
$$
Just by the definition, the set $S(p_1,p_2,\cdots)= \bigcup_{n=1}^\infty S(p_1,\cdots, p_n)$ is an IP-set.

The following fact has a straightforward proof.

\begin{lem}\label{lem:transitions}
If for some open sets $U,V$ and a point $x\in X$ we have $J\subset C_f(U,V)$ and $l\in N_f(x,U)$
then $l+s\in N_f(x,V)$ for every $s\in J$.
\end{lem}

Before proceeding, we also need the following technical lemma.

\begin{lem}\label{lem:ext_sum}
Assume that $A$ is a weakly mixing set of order $2$ for $(X,f)$. Let $U,V$ be open sets intersecting $A$ and $x\in U\cap A$, $y\in V\cap A$. If there are $p_1,\cdots, p_n\in \N$ satisfying
$$
S(p_1,\cdots, p_n)\subset N_f(x,V)\cap N_f(y,V)
$$
then there is an integer $p_{n+1}> \sum_{j=1}^n p_j$ and open sets $U',V'$ intersecting $A$ such that
$$
\overline{U'}\subset U, \overline{V'}\subset V\ \text{and}\ S(p_1,\cdots, p_n,p_{n+1})\subset C_f(U',V)\cap C_f(V',V).
$$
\end{lem}
\begin{proof}
By the assumptions, there are neighborhoods $W\subset V,W'\subset U$ of $y$ and $x$ respectively, such that $S(p_1,\cdots, p_n)\subset C_f(W',V)\cap C_f(W,V)$.
Since $A$ is weakly mixing of order 2, there are $l>\sum_{j=1}^n p_j$ and $x'\in A\cap W'$, $y'\in A\cap W$, such that
$$
l\in N_f(x', W)\cap N_f(y', W).
$$
But then, by Lemma~\ref{lem:transitions},  $l+s\in N_f(x',V)\cap N_f(y',V)$ for all $s\in S(p_1,\cdots, p_n)$. Putting $p_{n+1}=l$
we obtain that
$$
S(p_1,\cdots, p_{n+1})\subset N_f(x', V)\cap N_f(y', V).
$$
Now, it is enough to choose sufficiently small neighborhoods $U',V'$ of $x'$, $y'$ respectively and the proof is completed.
\end{proof}

The following fact is the main result of this section. It provides a characterization of transfer times of points in weakly mixing sets.

\begin{thm}\label{thm:TranIP}
Let $A$ be a weakly mixing set of order $2$ for $(X,f)$ and $U$ an open set intersecting $A$. Then there is $x\in U\cap A$ such that for every open set $V$ intersecting $A$
the set $N_f(x, V)$ contains an IP-set. In particular, for every open set $V$ intersecting $A$
the set $N_f(U\cap A, V)$ contains an IP-set.
\end{thm}
\begin{proof}
We choose a set $\set{x_i}_{i=1}^\infty \subset A$ dense in $A$. It suffices to prove that there is $x\in U\cap A$ such that for every $i$ and every $n>0$
the set $N_f(x, B(x_i,\frac{1}{n}))$ contains an IP-set.

Arrange sets $\set{B(x_i,\frac{1}{n})}_{i,n=1}^\infty$ into a sequence, let say $V_1,V_2,\cdots$ and order pairs $(i,j)$ by taking consecutive diagonals $\set{(i,j) : i+j=n}$ in $\N\times \N$, say
$$
(1,1)\prec (1,2)\prec (2,1)\prec (1,3)\prec (2,2)\prec (3,1)\prec (1,4)\prec \cdots.
$$
For technical reasons, we assume that $(0,0)\prec (1,1)$ and put as $U^{(0,0)}$ an open set intersecting $A$ such that $\overline{U^{(0,0)}}\subset U$ and diameter of $U^{(0,0)}$ is at most $1$. By the same reason we put $p^{(0)}_0=0$ and $V_i^{(0)}= V_i$ for each $i\in \N$.
We will perform an inductive construction with respect to the relation $\prec$ to construct:
\begin{enumerate}[(i)]
\item
a sequence of sets $U^{(i,j)}$ intersecting $A$ such that the diameter of $U^{(i,j)}$ is at most $\frac{1}{i+ j+ 1}$ and $\overline{U^{(i,j)}}\subset U^{(a,b)}$ if $(a,b)\prec (i,j)$;
\item for every $j=1,2,\cdots$ a sequence of sets $V^{(j)}_i$ intersecting $A$ such that $\overline{V^{(j)}_i}\subset V^{(j- 1)}_i$ (recall that $V_i^{(0)}= V_i$); and
\item a sequence of integers $p^{(j)}_i$ such that
$S(p^{(1)}_i,\cdots, p_i^{(j)})$ is a subset of $C_f(U^{(i,j)},V_i)$ $\cap C_f(V_i^{(j)},V_i)$ and $p^{(j)}_{i}> p^{(j- 1)}_{i}$ (recall that $p^{(0)}_0=0$).
\end{enumerate}

Consider a pair $(i,j)$ and assume that for the pair $(i',j')$ directly proceeding $(i,j)$ with respect to the relation $\prec$ we have already constructed a set $U^{(i',j')}$ such that $U^{(i',j')}\cap A\neq \emptyset$ and $\overline{U^{(i',j')}}\subset U^{(a,b)}$ for every $(a,b) \prec (i',j')$. We also assume that numbers $p^{(i_1)}_{j_1}$ have already been constructed for all $(i_1,j_1)\prec (i,j)$.

If $j=1$, then since $A$ is weakly mixing of order 2, we can find $x\in U^{(i',j')}\cap A$ and $y\in V_i\cap A$ such that there is $l\in N_f(x,V_i)\cap N_f(y,V_i)$.
Then, there are open sets $U^{(i,j)}\ni x$ and $V_i^{(1)}\ni y$ such that diameter of $U^{(i,j)}$ is at most $\frac{1}{i+ j+ 1}$ and $l\in C_f(U^{(i,j)},V_i)\cap C_f(V_i^{(1)},V_i), \overline{V^{(1)}_i}\subset V_i, \overline{U^{(i,j)}}\subset U^{(i',j')}$. We put $p_1^{(i)}=l$.

If $j>1$ then we have already constructed positive integers $p_i^{(1)}<p_i^{(2)}<\cdots < p_{i}^{(j-1)}$ and sets $V_i^{(1)}, \cdots, V_i^{(j-1)}$ intersecting $A$ such that $S(p_i^{(1)},\cdots, p_{i}^{(j-1)})\subset C_f(U^{(i,j-1)},V_i)\cap C_f(V_i^{(j-1)},V_i)$ and $\overline{ V_i^{(k)}}\subset V_i^{(k-1)}$ for all $k=1,\cdots, j-1$. Fix any $x\in U^{(i',j')}\cap A$ and $y\in V_i^{(j-1)}\cap A$. Applying Lemma~\ref{lem:ext_sum} we obtain an integer $p_i^{(j)}> \sum_{k=1}^{j-1}p_i^{(k)}$ and open sets
$\overline{U^{(i,j)}}\subset U^{(i',j')}$, $\overline{V_i^{(j)}}\subset V_i^{(j-1)}$, both intersecting $A$, such that diameter of $U^{(i,j)}$ is at most $\frac{1}{i+ j+ 1}$ and
$$
S(p^{(1)}_i,\cdots, p_i^{(j)})\subset C_f(U^{(i,j)},V_i)\cap C_f(V_i^{(j)},V_i).
$$

By the above inductive construction, we obtain a nested sequence of sets $U^{(i,j)}$ with respect to the relation $\prec$.
Now fix any
$$
z\in \bigcap_{i,j}\overline{U^{(i,j)}}\subset U.
$$
Note that each $U^{(i,j)}$ intersects the closed set $A$ and the diameter of $U^{(i,j)}$ is at most $\frac{1}{i+ j+ 1}$, so directly from the construction
we have that $z\in U\cap A$ and
$$
S(p^{(1)}_i,\cdots, p_i^{(j)})\subset N_f(z,V_i)
$$
 for any pair $(i, j)$. Therefore $N_f(z,V_i)$ contains the IP-set generated by the sequence $\set{p_i^{(j)}}_{j=1}^\infty$. But $i\in \N$ is arbitrary, and so the proof is completed.
\end{proof}

Now we are going to present two possible applications of the somewhat technical results in Theorem \ref{thm:TranIP}.

The following theorem was first proved by Furstenberg \cite[Theorem~II.3]{FurDA}. The main argument in his proof is an algebraic description of the structure of minimal distal systems \cite[Proposition II.9]{FurDA} which is an advanced result with a complicated proof (see also \cite{STDF}) and the fact that
the Cartesian product of a weakly mixing system with any minimal system is transitive \cite[Proposition II.11]{FurDA}. Here we offer another proof of Furstenberg's theorem using tool provided by Theorem \ref{thm:TranIP}. In particular our proof relies on dynamical properties of distal points rather than algebraic properties of distal minimal systems.
One of important tools in our proof (which is somehow hidden behind results of \cite[Theorem 9.11]{FurBook}) is Hindman's theorem.
While we definitely do not refer to algebraic classification of distal minimal systems, it is not completely honest to say that arguments in our proof are elementary.

\begin{cor}\label{cor:Fur}
If $(X,f)$ is weakly mixing then it is disjoint from any minimal distal system.
\end{cor}
\begin{proof}
If $X$ is a singleton then theorem holds. Therefore we can assume that $X$ has at least two points, and so $X$ is perfect.
Fix any minimal distal system $(Y,g)$ and let $J\subset X\times Y$ be a joining of $(X,f)$ and $(Y,g)$.
Let $x\in X$ be a point obtained by applying Theorem~\ref{thm:TranIP} with $A= U=X$. Since $J$ is a joining, there exists $y\in Y$ such that $(x,y)\in J$. It can be proved that a point $y$ is distal if and only if for each neighborhood $W$ of $y$, the set $N_g(y,W)$ has a non-empty intersection with any IP-set (mainly, because that each IP-set contains a set of return times of some recurrent point to its sufficiently small neighborhood), e.g. see \cite[Theorem 9.11]{FurBook}. It immediately implies that $N_f(x,V)\cap N_g(y,W)\neq \emptyset$ for any non-empty open set $V\subset X$ and any open neighborhood $W$ of $y$.
Since $V$ is arbitrary, taking a nested sequence of neighborhoods of $y$ and using the fact that $J$ is a closed set containing positive limit set of the pair $(x,y)$ under $f\times g$, we obtain the inclusion $X\times \set{y}\subset J$. But $y$ has a dense orbit in $Y$, hence we get $X\times Y\subset J$ which ends the proof.
%
\end{proof}

In \cite{BHM} the authors introduced scattering systems using topological complexity of open covers and proved that scattering systems are disjoint from all minimal distal systems \cite[Proposition 4.2]{BHM}. The proof of \cite[Proposition 4.2]{BHM} relies again on the algebraic description from \cite[Proposition II.9]{FurDA}. Each weakly mixing system is scattering \cite[Proposition 3.4]{BHM} and for a minimal system scattering and weak mixing are equivalent properties \cite[Proposition 3.8]{BHM}. Unfortunately, we do not know if our proof
can be adopted to prove \cite[Proposition 4.2]{BHM}, that is, we do not know how to extend our result to work also for non-minimal scattering systems.

The following property was introduced in \cite{PO_AIF}. A non-empty set $A\subset X$ has the \emph{property \textbf{(P)}} if for any
open set $U\subset X$ with $U\cap A\neq \emptyset$ there exists a point $x\in U\cap A$ and an integer $K>0$ such
that
$f^{nK}(x)\in U$ for all $n\in \N$. In other words, if we denote by $\F_{s, +}$ the family generated by $\{k \N: k\in \N\}$, then the subset $\emptyset\neq A\subset X$ has the property \textbf{(P)} then $N_f (U\cap A, U) \in \F_{s, +}$ for any open set $U\subset X$ with $U\cap A\neq \emptyset$.

In \cite{Ye_WPR, PO_AIF} it was proved that each weakly mixing system with dense distal points is disjoint from all minimal systems and all points with a dense orbit in such a system are weakly product recurrent. Another result of \cite{PO_AIF} shows that each weakly mixing set with the property \textbf{(P)} contains a residual subset of weakly product recurrent points. It is also possible to prove that (see \cite{Huang}) every weakly mixing system with the property \textbf{(P)} is disjoint from all minimal systems. Then, a natural claim is that every weakly mixing set (of order $2$) with dense distal points should contain weakly product recurrent points. However, it is much harder to synchronize with distal points than with uniformly recurrent points, and because of this difficulty the above claim was left open in \cite{PO_AIF}. Now, Theorem~\ref{thm:TranIP} provides a tool which can be used to finally prove the above mentioned fact, completing the results of \cite{PO_AIF}.

\begin{cor}\label{cor:distal_wpr}
Let $A\subset X$ be a 
weakly mixing set of order $2$ for $(X,f)$. If additionally distal points are dense in $A$ then $A$ contains a residual subset of points which are weakly product recurrent but not product recurrent.
\end{cor}
\begin{proof}
Fix any sequence $\set{x_i}_{i=1}^\infty\subset A$ dense in $A$ which consists of distal points in $A$.
We consider the subset $D\subset A$ defined as follows: a point $x\in A$ belongs also to $D$ if for any point $x_i$ and any open set $U\ni x_i$
we have $N_f(x,U)\cap N_f(x_i,U)\neq \emptyset$.

First, let us prove that $D$ is residual in $A$.
Obviously $D$ is a $G_\delta$ subset, as $D=\bigcap_{i,j=1}^\infty D_{i,j}$ where $D_{i,j}$ is the open set consisting of all points $x\in A$ such that
$$N_f(x,B(x_i,\frac{1}{j})) \cap N_f(x_i, B(x_i,\frac{1}{j}))\neq \emptyset.$$ It remains to show that each $D_{i,j}$ is a dense subset of $A$.
Fix any $i,j>0$ and
any open set $W$ intersecting $A$. Let $x\in W\cap A$ be a point obtained from Theorem~\ref{thm:TranIP}.
By the choice of $x$ there is an IP-set $P\subset N_f(x,B(x_i,\frac{1}{j}))$, and so $P\cap N_f(x_i,B(x_i,\frac{1}{j}))\neq \emptyset$ as $x_i$ is a distal point (we apply again the characterization of product recurrence \cite[Theorem 9.11]{FurBook}).
Thus $D_{i, j}$ is dense in $A$, and then $D$ is residual in $A$.

By the construction, $\overline{\orbp(x,f)}\supset A$ for any $x\in D$. 
It remains to show that each point in $D$ is weakly product recurrent, since it is easy to verify that it cannot be product recurrent. Namely, if $x$ is product recurrent which is equivalent to say that $x\in D$ is distal, then a contradiction to the fact of $A\subset \overline{\orbp(x,f)}$, since it is easy to verify that in that case $\overline{\orbp(x,f)}$ contains a point (other than $x$ itself) proximal to $x$. In fact, by the construction of $D$ each point $x_i$ is proximal to $x$.

To finish the proof, fix any $x\in D$ and any uniformly recurrent point $y$ in a dynamical system $(Y,g)$, and next fix any open neighborhood $U\subset X$ of $x$ and any open neighborhood $V\subset Y$ of $y$. There is $i\in \N$ such that $x_i\in U$, and so by \cite[Theorem 9.11]{FurBook} the pair $(x_i,y)$ is uniformly recurrent. We choose $\eps>0$ with $B(x_i,2\eps)\subset U$. Since $(x_i, y)$ is uniformly recurrent, the set $N_f(x_i,B(x_i,\eps))\cap N_g(y,V)$ is syndetic, and so there is $M\in \N$ such that it intersects any block of consecutive $M$ integers in $\N$. Now we choose an open set $W\ni x_i$ such that diameter of each $f^j(W)$ is smaller than $\eps$ for $j=0, 1, \cdots, M$.
But $x\in D$, so by the definition of $D$ there is
$t\in
N_f(x,W)\cap N_f(x_i,W)$. Then there is $j\in \{0, 1, \cdots, M\}$ such that
$t+j\in N_f(x_i,B(x_i,\eps))\cap N_g(y,V)$ and additionally $d(f^{t+j}(x),f^{t+j}(x_i))\le\eps$, because the diameter of $f^j(W)$ is smaller than $\eps$ and both $f^t (x)$ and $f^t (x_i)$ are contained in $W$. This gives
$f^{t+j}(x)\in B(x_i,2\eps)\subset U$ and therefore $N_f(x,U)\cap N_f(y,V)\neq \emptyset$, which completes the proof.
\end{proof}

As another application of Theorem~\ref{thm:TranIP}, we have the following result, which was suggested to us by the anonymous referee of the paper.

\begin{cor} \label{referee}
Let $A\subset X$ be a 
weakly mixing set of order $2$ for $(X,f)$. If additionally distal points are dense in $A$ then $A$ is weakly mixing of all orders.
\end{cor}
\begin{proof}
Fix any $n> 2$ and any open sets $U_1, V_1, \cdots, U_n, V_n$ intersecting $A$. It suffices to prove that $\bigcap_{i=1}^n N_f (U_i\cap A, V_i)\neq \emptyset$.

First, by assumptions we may take distal points $x_i\in V_i\cap A$ for each $i= 2, \cdots, n$. By the definitions, clearly $(x_2, \cdots, x_n)$
is distal for $f\times \cdots \times f$, and hence the following set (which is a subset of $\bigcap_{i= 2}^n N_f (V_i\cap A,V_i)$):
$$
N_{f\times \cdots\times f} ((x_2,\cdots,x_n), V_2\times \cdots \times V_n)\
$$
has a non-empty intersection with any IP-set (the same arguments as in Corollary \ref{cor:Fur} apply).
As a consequence of Theorem~\ref{thm:TranIP} we obtain that $N_f (U_1\cap A, V_1)$ contains an IP-set, thus
$$
N_f(U_1\cap A,V_1)\cap \bigcap_{i= 2}^n N_f (V_i\cap A,V_i)\neq \emptyset.
$$

Now, let us assume that, for some $1\leq j<n$ we have
$$
\bigcap_{i= 1}^j N_f (U_i\cap A,V_i)\cap \bigcap_{l= j+ 1}^n N_f (V_l\cap A, V_l)\neq \emptyset.
$$
Then there is $k>0$ and open sets $U_1', \cdots, U_n'$ intersecting $A$ such that $f^k(U_i')\subset V_i$ for each $i=1,\cdots,n$, and  $U_i'\subset U_i$ for $1\leq i\leq j$ and $U_l'\subset V_l$ for $l= j+ 1, \cdots, n$. Repeating previous arguments we obtain
that the following set of integers is non-empty:
$$
M\doteq N_f(U_{j+ 1}\cap A, U_{j+ 1}')\cap \bigcap_{i\in \{1, \cdots, j, j+ 2, \cdots, n\}} N_f (U_i'\cap A,U_i')\neq \emptyset.
$$
If we take any $m\in M$ then it is easy to check that
$$
m+ k\in N_f(U_{j+ 1}\cap A, f^k (U_{j+ 1}'))\cap \bigcap_{i\in \{1, \cdots, j, j+ 2, \cdots, n\}} N_f (U_i'\cap A, f^k (U_i')),
$$
and so
$$
\bigcap_{i= 1}^{j+ 1} N_f (U_i\cap A,V_i)\cap \bigcap_{l= j+ 2}^n N_f (V_l\cap A, V_l)\neq \emptyset.
$$
Hence,
by induction we eventually obtain that $\bigcap_{i=1}^n N_f(U_i\cap A,V_i)\neq \emptyset$, completing the proof.
\end{proof}

The following example shows that there are systems fulfilling assumptions of Corollary~\ref{cor:distal_wpr} which can not be covered by results of \cite{PO_AIF}.

\begin{exmp}
Let $R$ be an irrational rotation of the unit circle $\Cir$ and let $T$ be the standard tent map on the unit interval.
Consider $F=R\times T$ acting on $X=\Cir \times [0,1]$. Note that the set $A=\set{x}\times [0,1]$
is a weakly mixing set for any $x\in \Cir$. To see this, fix any open sets $W_1,\cdots, W_n\subset X$ intersecting $A$.
There exist open intervals $V_1,\cdots, V_n\subset [0,1]$ and $U\subset \Cir$ such that $U\times V_i\subset W_i$ and $U\times V_i\cap A\neq \emptyset$ (i.e. $x\in U$)
for all $i=1,\cdots,n$. But $T$ is the tent map, so there exists $N>0$ such that $T^k(V_i)=[0,1]$ for every $k>N$ and $i= 1, \cdots, n$. There also exists $K>N$ such that $R^K(x)\in U$.
This shows that $F^K(W_i\cap A)\supset \set{R^K(x)}\times [0,1]$ and so $F^K(W_i\cap A)\cap W_j\neq \emptyset$ for any $i,j=1,\cdots,n$.
Indeed $A$ is a weakly mixing set.

Additionally $A$ contains a dense set of distal points, since any periodic point of $T$ contained in the fibre $A$ will generate a distal point for $F$. This shows, by Corollary~\ref{cor:distal_wpr}, that there are many weakly product recurrent but not product recurrent points for $F$.
But there is no a set in $X$ which can fulfil the regular condition of return times required by the property \textbf{(P)}.
Simply, if we fix any $z\in X$ then for every open set $U\times V\subset X$ and every $K>0$ there is $n\in \N$ such that $\pi_1(F^{nK}(z))=R^{nK}(\pi_1(z))\not\in U$,
where $\pi_1$ is the projection onto the first coordinate, and so $F^{nK}(z)\not\in U\times V$.
\end{exmp}

The next example (constructed in Theorem~\ref{thm10}) shows that assumptions of Corollary~\ref{cor:distal_wpr} can not be weakened too much. Strictly speaking, it is not enough if a weakly mixing set is contained in the closure of distal or even regular uniformly recurrent points.

First, let us recall some basic facts on Toeplitz flows (a more detailed exposition on properties of Toeplitz flows can be found in \cite{DO}). Suppose that $X_\omega$ is a Toeplitz flow,
that is $X_\omega=\overline{\orbp(\omega,\sigma)}$ for some Toeplitz sequence $\omega$.
For $x\in X_\omega$ and an integer $l>0$, we denote
$$
\Per_l(x)=\set{n\in \N\; : \; x(n)=x(n+kl) \text{ for every }k=1,2,\cdots}.
$$
As $\omega$ is a Toeplitz sequence then $\N=\bigcup_{m}\Per_m(\omega)$. By an \emph{essential period} of $\omega$ we mean any $s$ such that $\Per_s(\omega) \neq \emptyset$ and does not coincide
with $\Per_{k} (\omega)$ for any $k < s$. A period structure of $\omega$ is any sequence
$\textbf{s} = \set{s_m}_{m=1}^\infty$ of essential periods such that each $s_m$ divides $s_{m+1}$ and
$\N=\bigcup_{m}\Per_{s_m}(\omega)$. It is known that a periodic structure always exists, and $X_\omega$ is an almost 1-1 extension of the odometer (inverse limit with addition $(\text{mod }s_m)$ on each coordinate):
$$G_\mathbf{s}= \overleftarrow{\lim_m} \ \{0, 1, \cdots, s_m- 1\}$$
which is well defined as the subset consisting of all ``paths"
$$(j_1, j_2, \cdots)\in \prod_m \ \{0, 1, \cdots, s_m- 1\}$$ such that $j_{m+ 1}= j_m$ (mod $s_m$) for each $m\in \N$. So we will always assume that a periodic structure is fixed together with
the factor map $\pi_\omega\colon X_\omega \ra G_\mathbf{s}$. If $x\in X_\omega$ is not a Toeplitz sequence itself, then it may happen that periodic parts do not cover the whole $\N$, that is the \emph{aperiodic part of $x$}
$$
\Aper(x):=\N \setminus \bigcup_{m=1}^\infty \Per_{s_m}(x)
$$
is non-empty.
However it can also be proved that if we fix any $\mathbf{j}\in G_\mathbf{s}$ then $\Aper(x)=\Aper(y)$ for every $x,y\in \pi_\omega^{-1}(\mathbf{j})$ (for further properties of the set $\Aper(x)$  we refer again to \cite{DO}), in particular we can write $\Aper(\mathbf{j}):=\Aper(x)$. Assume that $\Aper(x)$ is infinite and enumerate its elements, say $\Aper(x)=\set{0\leq n_1<n_2<\cdots}$. Then by the \emph{aperiodic readout} of $x$ we mean the sequence
$$
y= x|_{\Aper(x)}= x[n_1]x[n_2]\cdots.
$$
In other words, we read and write down symbols along $\Aper(x)$. For every $\mathbf{j}\in G_\mathbf{s}$ with $\#\Aper(\mathbf{j})=\infty$, let $Y_\mathbf{j}$ be the set of all possible aperiodic readouts of elements in $\pi_\omega^{-1}(\mathbf{j})$.

\begin{thm}\label{thm10}
There is a minimal dynamical system with dense distal points, such that it contains weakly mixing sets and
any of its weakly mixing sets does not contain weakly product recurrent points.
\end{thm}
\begin{proof}
Let $X_\omega$ be a Toeplitz flow which is an extension of an odometer $G_\textbf{s}$ via the factor map $\pi_\omega\colon X_\omega \ra G_\mathbf{s}$, such that $Y_{\mathbf{j}}=\Sigma_2$ for every $\mathbf{j}\in G_\mathbf{s}$ with $\#\Aper(\mathbf{j})=\infty$, where $\Sigma_2$ is the two-sided infinite sequence over the alphabet $\set{0,1}$ (one of possible constructions of such a flow is explicitly described in \cite{Kurka}). Since $X_\omega$ is an almost 1-1 extension of an odometer, the set of all distal points is dense in $X_\omega$. It can be proved that $X_\omega$ has positive topological entropy, and so by results of \cite{BH, PZ} it has at least one weakly mixing subset.

Observe that each odometer is equicontinuous and so contains no weakly mixing subsets. Now, if $A$ is a weakly mixing set of $X_\omega$ then it must be contained in a fibre $\pi_\omega^{- 1} (\mathbf{j})$ with $\# Y_{\mathbf{j}}=\infty$, which equivalently means that $\#\Aper(\mathbf{j})=\infty$. But then aperiodic part of elements of the fiber $\pi_\omega^{- 1} (\mathbf{j})$ equals to $\Sigma_2$.
Aperiodic part occurs at the same positions in each element of $A$ and furthermore any two points in the same fibre have the same symbols on the periodic part.

For any $x\in A$ let $y\in \pi_\omega^{- 1} (\mathbf{j})$ be such that $y\neq x$ and $(x|_{\Aper(\mathbf{j})})_{[1,\infty)}=(y|_{\Aper(\mathbf{j})})_{[1,\infty)}$. In other words, $x$ and $y$
are defined by points in $\Sigma_2$ which differ on the first position.
It means that $x_{[i,i+N]}=y_{[i,i+N]}$ for every $N\in \N$ and all $i>K$ where $K\in \N$ depends only on $\Aper(\mathbf{j})$ (in particular, it is independent of $N\in \N$). If we fix sufficiently small neighborhoods $U,V$ of $x,y$, respectively, then $x,y$ cannot return to $U$ and $V$ respectively synchronously (in fact, it is enough that $U,V$ are defined by cylinders of words longer than $K$).
This proves that $(x, y)$ is not recurrent, and so $x$ is not weakly product recurrent, since $y$ is uniformly recurrent as a member of a minimal system.
\end{proof}

\section{Weakly mixing sets and proximality}

As shown by results of \cite{PZ, PZpreprint} the limit behavior of points in a weakly mixing set can be quite complex. In this section we shall present some further discussions along this line.

We begin this section with the following easy observation.

\begin{lem}\label{thm:const_WM}
Let $A$ be a weakly mixing set and $x_1,\cdots, x_n\in A$. For every sequence $U_1,\cdots, U_n$ of open neighborhoods of points $x_1, \cdots, x_n$, respectively, there is a weakly mixing set $B$ such that $x_i\in B$ for all $i=1,\cdots,n$.
\end{lem}
\begin{proof}
The proof will be finished by a simple construction using the definition. We leave details to the reader.
\end{proof}

Recall that for $x\in X$, its \emph{(positive) limit set} $\omega_f (x)$ is defined as the set of all points $y\in X$ such that $\lim_{k\rightarrow \infty} f^{n_k} x= y$ for some increasing sequence $\{n_1< n_2< \cdots\}$ in $\N$. The following example shows that points in weakly mixing sets do not have
to have good recurrence properties.

\begin{exmp}
There is a weakly mixing set $A$ such that $\omega_f(x)\cap A=\emptyset$ for some $x\in A$.
\end{exmp}
\begin{proof}
Take any non-minimal weakly mixing system $(X, f)$ containing a non-recurrent point $x$ with a minimal limit set. A trivial example is the two-sided full shift over the alphabet $\set{0,1}$. Next, let $U=X\setminus \overline{B(\omega_f(x),\eps)}$ where $0< \eps< \frac{1}{2} \dist(\omega_f(x),x)$, where as usual $\dist(\omega_f(x),x)$ is the distance between the set $\omega_f (x)$ and the point $x$. Then $x\in U$ and so it suffices to apply Lemma~\ref{thm:const_WM} with $x_1=x$ and $U_1=U$ obtaining a weakly mixing set $B\subset U$ such that $x\in B$. But then $\dist(B,\omega_f(x))>\eps$
and the result follows.
\end{proof}

By the above example we cannot guarantee that a member of a weakly mixing set $A$ will return to a neighborhood of $A$. In particular, points obtained
by application of Theorem~\ref{thm:TranIP} are very special. Despite of this difficulty, we still can guarantee some degree of synchronization of trajectories of points in $A$.
In the case of $X=A$ the following lemma follows directly from the definition. In the local case $A\varsubsetneq X$ we have to perform a more careful approximation of desired points.

\begin{lem}\label{lem:local_sync}
Let $A$ be a weakly mixing set and $x\in A$. Then for every open set $U$ intersecting $A$ and each $\eps>0$ there are $n\in \N$ and $y\in U\cap A$ with $d(f^n(x),f^n(y))\le \eps$.
\end{lem}
\begin{proof}
Fix any $x\in A$ and any $\eps>0$.
Let $U$ be an open set intersecting $A$ and let $\{V_1,\cdots, V_k\}$ be a cover of $X$ consisting of open sets with diameters less than $\eps$.
We will perform a construction of $y$ and $n$ in a finite number of steps (at most $k$).

By weak mixing of $A$ there are $i_1\in \{1, \cdots, k\}$, and integer $n_1>0$ and an open set $U_1^{(1)}\subset U$ intersecting $A$ such that
$f^{n_1}(U_1^{(1)})\subset V_{i_1}$, where $x\in V_{i_1}$.

Next, assume that for some $m\geq 1$ we have constructed open sets $U_1^{(m)},\cdots, U_m^{(m)}$ $\subset U$ intersecting $A$, an integer $n_m>0$ and pairwise distinct integers $i_1, \cdots, i_m\in \{1, \cdots, k\}$ such that
$f^{n_m}(U_j^{(m)})\subset V_{i_j}$ for each $j= 1, \cdots, m$. If $f^{n_m}(x)\not\in \bigcup_{j=1}^m V_{i_j}$, then we can choose $i_{m+1}\in \set{1, \cdots, k}\setminus \set{i_1,\cdots, i_m}$
and an open set $U_{m+1}^{(m)}$ containing $x$ (and so intersecting $A$) such that $f^{n_m}(U_{m+1}^{(m)})\subset V_{i_{m+1}}$. By weak mixing of $A$ there are open sets $U_j^{(m+1)}\subset U$ intersecting $A$ and $s>0$ such that $f^s(U_j^{(m+1)})\subset U_j^{(m)}$ for each $j=1,\cdots,m+1$. Now, if we put $n_{m+1}=n_m+s$ then, for each $j=1,\cdots, m+1$ we obtain that:
$$
f^{n_{m+1}}(U_j^{(m+1)})=f^{n_{m}}(f^s(U_j^{(m+1)}))\subset f^{n_{m}}(U_j^{(m)})\subset V_{i_j}.
$$

Obviously at some step $m\le k$ we cannot extend sequence $i_1, \cdots, i_m$ any further by the above procedure.
Hence, we have that $f^{n_m}(x)\in \bigcup_{j=1}^m V_{i_j}$, in particular $f^{n_m}(x)\in V_{i_l}$ for some $l\in \set{1, \cdots, k}$. But then by the construction $f^{n_m}(U_l^{(m)})\subset V_{i_l}$, and so if we fix any $y\in U_l^{(m)}\cap A\subset U\cap A$ then
$f^{n_m}(y),f^{n_m}(x)\in V_{i_l}$. We have just proved that
$d(f^{n_m}(y),f^{n_m}(x))< \eps$, as the diameter of $V_{i_l}$ is less than $\eps$,
which ends the proof.
\end{proof}

First, it was proved in \cite{KR69} that for weakly mixing systems the set of points $x$ at which
$\Prox (f) (x)$ is residual in $X$ is itself residual in $X$, that is, for $\Prox (f)$ there is a residual set
of parameters where sections are also residual. Later it was proved by Furstenberg in \cite{FurBook}
that $\Prox (f) (x)$ is residual for any $x$, provided that $(X, f)$ is minimal and weakly mixing.
Finally, Akin and Kolyada proved in \cite{AK03} that in a weakly mixing system $\Prox (f)(x)$ is dense
for every point $x\in X$ (hence residual, because it is always a $G_\delta$ set).

Now, we have enough
tools at hand to prove yet another extension of these classical results.

\begin{thm}\label{thm:proxcell}
For every weakly mixing set $A$ and every $x\in A$ the set $\Prox(f)(x)\cap A$ is residual in $A$.
\end{thm}
\begin{proof}
Fix any $x\in A$ and $\eps>0$.
By Lemma~\ref{lem:local_sync}, if we consider the set $A_\eps$ consisting of all points $y\in A$ such that $d(f^n(x),f^n(y))<\eps$ for some $n>0$, then $A_\eps$ is a dense subset of $A$.
But it is also easy to verify that $A_\eps$ is an open subset of $A$. This, by the inclusion
$$
\Prox(f)(x)\cap A \supset \bigcap_{n=1}^\infty A_{\frac{1}{n}},
$$
proves that $\Prox(f)(x)$ is residual in $A$ which completes the proof.
\end{proof}

\begin{rem}
Theorem~\ref{thm:proxcell} provides another tool that can be used to prove Corollary~\ref{cor:distal_wpr} similarly
in the special case of weakly mixing sets. Unfortunately, so far we do not have any evidence that these cases are different.
\end{rem}

\section*{Acknowledgements}

The authors would like to thank Xiangdong Ye and Wen Huang for their valuable discussions and important comments about the preprint version of the manuscript.

Just before submission of the manuscript, Jian Li informed the authors
that he introduced arguments similar to the proof of Corollary~\ref{cor:Fur} in his proof
of \cite[Lemma 5.1]{lijian} which gives a more general tool than original Furstenberg's theorem. This, after some adjustment of
terminology, leads to a proof analogous to the one presented in Corollary~\ref{cor:Fur}. We are grateful to him for this remark.
We also express many thanks to the anonymous referee, whose remarks resulted in substantial
improvements to this paper, in particular, for his contribution to Corollary \ref{referee}.

The first author was supported by the Marie Curie European Reintegration
Grant of the European Commission under grant agreement no. PERG08-GA-2010-
272297. The second author was supported by FANEDD (grant no. 201018) and NSFC (gtant no. 11271078).

\end{document}